\documentclass[12pt,a4paper]{amsart}

\title[Products of residue currents]{Various approaches to products of residue currents}

\author{Richard L\"{a}rk\"{a}ng \& H\aa kan Samuelsson Kalm}
\address{R.\ L\"{a}rk\"{a}ng, H.\ Samuelsson, Department of Mathematical Sciences, Division of Mathematics, 
University of Gothenburg and Chalmers University of Technology, SE-412 96 G\"{o}teborg, Sweden}
\email{larkang@chalmers.se, hasam@chalmers.se}
\date{\today}

\usepackage{amsmath}
\usepackage{amsthm,amssymb,latexsym}
\usepackage{t1enc}
\usepackage{fancyhdr}
\usepackage{mathrsfs}
\usepackage{hyphenat}

\newtheorem{proposition}{Proposition}[]
\newtheorem{theorem}[proposition]{Theorem}
\newtheorem{lemma}[proposition]{Lemma}

\theoremstyle{definition}
\newtheorem{definition}[proposition]{Definition}

\newtheorem{remark}[proposition]{Remark}

\newcommand{\C}{\mathbb{C}}
\newcommand{\debar}{\bar{\partial}}
\newcommand{\D}{\mathscr{D}}

\newcommand{\Q}{\mathbb{Q}}

\newcommand{\CP}{\mathbb{CP}}

\newcommand{\B}{\mathbb{B}}

\def\newop#1{\expandafter\def\csname #1\endcsname{\mathop{\rm #1}\nolimits}}
\newop{rank}
\newop{supp}

\begin{document}
\nocite{*}
\bibliographystyle{plain}

\begin{abstract}
We describe various approaches to Coleff-Herrera products of residue currents $R^j$
(of Cauchy-Fantappi\`e-Leray type) associated to holomorphic
mappings $f_j$. More precisely, we study
to which extent (exterior) products of natural regularizations of the individual 
currents $R^j$ yield regularizations of the corresponding Coleff-Herrera products. Our results hold globally on 
an arbitrary pure-dimensional complex space.
\end{abstract}

\maketitle
\thispagestyle{empty}

\section{Introduction}

\bigskip

Let $f$ be a holomorphic function defined on the unit ball $\B \subset \C^n$.
If $f$ is a monomial it is elementary to show, e.g., by integrations by parts or by a Taylor expansion,
that the principal value current
$\varphi \mapsto \lim_{\epsilon \to 0}\int_{|f|^2>\epsilon}\varphi/f$, $\varphi \in \D_{n,n}(\B)$, exists and defines a 
$(0,0)$-current $1/f$ that we also denote by $U^f$. From Hironaka's theorem it then follows that such limits exist 
for general $f$ and also that $\B$ may 
be replaced by a complex space, \cite{HL}. 
The $\debar$-image, $R^f:=\debar (1/f)$, is the residue current of $f$ and by Stokes' theorem it is 
given by $\varphi \mapsto \lim_{\epsilon \to 0}\int_{|f|^2=\epsilon}\varphi/f$, $\varphi \in \D_{n,n-1}(\B)$. 
It has the useful property that its 
annihilator ideal is equal to the principal ideal $\langle f \rangle$ and, moreover, it gives a factorization of Lelong's
integration current; $2\pi i [f=0] = \debar(1/f) \wedge df$. 

There are (at least) two natural ways of regularizing $U^f$ and $R^f$. If $\lambda \in \C$ and 
$\mathfrak{Re}\, \lambda \gg 0$, then $\lambda \mapsto \int \varphi \, |f|^{2\lambda}/f$ is holomorphic for any
test form $\varphi$. It is well known (cf., Lemma~\ref{pmmultlemma}) that the current-valued map
$\lambda \mapsto |f|^{2\lambda}/f=:U^{f, \lambda}$ has a meromorphic extension to $\C$ with poles contained in the set of
negative rational numbers and that the value at $\lambda=0$ is $U^f$. It follows that 
$\lambda \mapsto \debar |f|^{2\lambda}/f=:R^{f,\lambda}$ is meromorphic in $\C$, analytic in a half space containing the origin, 
and that the value at the origin is $R^f$. The technique of using analytic continuation in 
residue current theory has its roots in the work of Atiyah, \cite{At},
and Bernstein-Gel'fand, \cite{BG}. In the context of residue currents 
it has been developed by several authors, e.g., Barlet-Maire, \cite{BaMa}, Yger, \cite{Y},
Berenstein-Gay-Yger,\cite{BGY}, Passare-Tsikh, \cite{PTCanad},
and recently by the second author in \cite{HasamArkiv}.
The second regularization method, inspired by Passare, \cite{PCrelle}, 
is more explicit and concrete; $U^f$ and $R^f$ are obtained as weak limits 
of explicit smooth forms. Let $\chi$ be a smooth regularization of the characteristic function $\mathbf{1}_{[1,\infty)}$
and let $U^{f,\epsilon}:=\chi(|f|^2/\epsilon)/f$ and $R^{f,\epsilon}:= \debar\chi(|f|^2/\epsilon)/f$. Then (see, e.g., 
\cite{PCrelle})  
$U^f=\lim_{\epsilon \to 0^+}U^{f,\epsilon}$ and 
$R^f=\lim_{\epsilon\to 0^+} R^{f,\epsilon}$ in the sense of currents. Notice that the original definition 
mentioned above corresponds to $\chi=\mathbf{1}_{[1,\infty)}$.

If $f$ is a tuple of functions or a section of a vector bundle there are natural analogues of the currents
$1/f$ and $\debar(1/f)$ introduced in \cite{PTY} and \cite{MatsAB}. The construction of these more general
currents, still denoted $U^f$ and $R^f$, 
is based on Bochner-Martinelli and Cauchy-Fantappi\`{e}-Leray type formulas; see Section~\ref{formulering} for details.
In this paper we consider products of regularized currents of this kind and we investigate their limit behavior. It turns out
that both the $\lambda$-approach and the $\epsilon$-approach yield the same current as the classical Coleff-Herrera approach.

\begin{center}
---
\end{center}

\noindent Let $Z$ be a reduced complex space of pure dimension $n$, let $E_1,\ldots,E_p$ be hermitian holomorphic vector 
bundles over $Z$, and let $f_j$ be a holomorphic section of $E_j^*$. Then $U^{f_j}=:U^j$ and $R^{f_j}=:R^j$ become
currents with values in $\Lambda E_j$; 
if $\textrm{rank}\, E_j=1$ then $U^j$ is the principal value current associated with the meromorphic section $1/f_j$ of $E_j$ and 
$R^j=\debar U^j$. In complete analogy with the regularization methods discussed above we have
\begin{equation*}
U^j=U^{j,\lambda}\big{|}_{\lambda=0}=\lim_{\epsilon\to 0^+} U^{j,\epsilon}  \quad \textrm{and} \quad
R^{j}=R^{j,\lambda}\big{|}_{\lambda=0} = \lim_{\epsilon\to 0^+} R^{j,\epsilon},
\end{equation*}
see Section~\ref{formulering}. We define products of the $R^j$ (for simplicity we restrict attention to such products in 
this section) recursively as follows:
Having defined $R^{k-1}\wedge \cdots \wedge R^1$ it turns out (see \cite{AWCrelle} or Section~\ref{formulering}) that 
\begin{equation*}
\lambda \mapsto R^{k,\lambda}\wedge R^{k-1}\wedge \cdots \wedge R^1
\end{equation*}
has an analytic continuation to a neighborhood of $\lambda=0$ and we define 
$R^{k}\wedge \cdots \wedge R^1$ as the value at $\lambda=0$. From the proof of Proposition~5.4 in \cite{ASWY}
it follows that one can compute the product in the following way: If $a_1 >\cdots >a_p>0$ are integers then
\begin{equation*}
R^{p}\wedge \cdots \wedge R^1=
R^{p,\lambda^{a_p}}\wedge \cdots \wedge R^{1,\lambda^{a_1}}\big{|}_{\lambda=0}.
\end{equation*}
That is, the recursive definition can be replaced 
by the evaluation of a one-variable analytic (current valued) 
function at the origin; we just have to make sure that $\lambda^{a_1}$
tends to zero much faster than $\lambda^{a_2}$ and so on.

We now consider the smooth form $R^{p, \epsilon_p}\wedge \cdots\wedge R^{1,\epsilon_1}$ and limits of it 
of the following kind:
\begin{definition}\label{limitdef}
Let $\vartheta$ be a function defined on $(0,\infty)^p$. We let
\begin{equation*}
\lim_{\epsilon_1 \ll \cdots \ll \epsilon_p\to 0}\vartheta(\epsilon_1,\ldots,\epsilon_p)
\end{equation*}
denote the limit (if it exists and is well-defined) 
of $\vartheta$ along any path $\delta\mapsto \epsilon (\delta)$ towards the origin
such that for all $\ell\in \mathbb{N}$ and $j=2,\ldots,p$ there are positive constants $C_{j\ell}$ such that
$\epsilon_{j-1}(\delta) \leq C_{j\ell}\, \epsilon_j^{\ell}(\delta)$.
Here, we extend the domain of definition of $\vartheta$ to points $(0,\dots,0,\epsilon_{m+1},\dots,\epsilon_p)$,
where $\epsilon_{m+1},\dots,\epsilon_p > 0$, by defining
\begin{equation*}
    \vartheta(0,\dots,0,\epsilon_{m+1},\dots,\epsilon_p) = \lim_{\epsilon_m\to 0}\dots\lim_{\epsilon_1\to 0}
    \vartheta(\epsilon_1,\dots,\epsilon_m,\epsilon_{m+1},\dots,\epsilon_p).
\end{equation*}
if the limits exist.
\end{definition}

Recall that $(\epsilon_1,\ldots,\epsilon_p)$ tends to zero along an {\em admissible paths} in the sense of 
Coleff-Herrera, \cite{CH}, if it tends to zero along a path inside $(0,\infty)^p$ such that 
$\epsilon_{j-1}/\epsilon_j^{\ell}\to 0$ for all $\ell \in \mathbb{N}$ and $j=2,\ldots,p$. 
The limits in Definition~\ref{limitdef} are (slightly) more general since, e.g., $\epsilon_1$ is allowed 
to attain the value $0$ before the other $\epsilon_j$ go to zero.
In particular, it thus includes the iterated limit letting $\epsilon_k \to 0$ one at a time.
The following theorem is a special case of 
Theorem~\ref{main} below. The proof 
shares many similarities with the proof of \cite[Proposition 1]{PCrelle}
(even though the statements differ).
However, in our case, extra
technical difficulties arise since the bundles $E_j$ may have non-trivial metrics.

\begin{theorem}\label{sats1}
In the sense of currents we have
\begin{equation*}
R^p\wedge \cdots \wedge R^1= \lim_{\epsilon_1 \ll \cdots \ll \epsilon_p\to 0}
R^{p,\epsilon_p}\wedge \cdots \wedge R^{1,\epsilon_1}.
\end{equation*}
\end{theorem}

To connect with the classical Coleff-Herrera approach, assume temporarily that $\textrm{rank}\, E_j=1$, $j=1,\ldots,p$,
so that $R^j=\debar(1/f_j)$. Then Theorem~\ref{sats1} says that for any test form $\varphi$ of bidegree $(n,n-p)$ 
\begin{equation*}
\debar \frac{1}{f_p}\wedge \cdots \wedge \debar \frac{1}{f_1} . \varphi =
\lim_{\epsilon_1 \ll \cdots \ll \epsilon_p\to 0} \int_Z
\frac{\debar\chi^{\epsilon_p}}{f_p}\wedge \cdots \wedge \frac{\debar\chi^{\epsilon_1}}{f_1}\wedge \varphi, 
\end{equation*}
where $\chi^{\epsilon_j}=\chi(|f_j|^2/\epsilon_j)$. We will refer to the integral on the right hand side as 
the residue integral and denote it by $\mathcal{I}_f^{\varphi}(\epsilon)$. If the $\chi$-functions tend
to $\mathbf{1}_{[1,\infty)}$ (for a fixed generic $\epsilon \in (0,\infty)^p$) then $\mathcal{I}_f^{\varphi}(\epsilon)$
tends to Coleff-Herrera's original
residue integral
\begin{equation}\label{CHresint}
I_f^{\varphi}(\epsilon)=\int_{T(\epsilon)}\varphi/(f_1\cdots f_p),
\end{equation}   
where $T(\epsilon)=\cap_1^p\{|f_j|^2=\epsilon_j\}$ is oriented as the distinguished boundary of the corresponding 
polyhedron. In \cite{CH} Coleff and Herrera prove that the limit of $I_f^{\varphi}(\epsilon)$ along an admissible path
exists and defines a current, the nowadays called {\em Coleff-Herrera product}.
We show (see Theorem~\ref{main}) that the Coleff-Herrera product equals the product 
$\debar(1/f_p)\wedge \cdots \wedge \debar(1/f_1)$; this is folklore but to our knowledge not completely proved 
before (except in the case of complete intersection when it follows from \cite{PCrelle} and \cite{Plambda}
together with \cite{HasamArkiv}).

\smallskip

A result much in the same spirit was proven by Passare in \cite{Plambda}, where he relates 
the original Coleff-Herrera product
to residue currents defined by $\lambda$-regularizations. Passare considers
the regularization
\begin{equation}
    \label{eqpdef}
    \left.\frac{\debar |f_p|^{2\lambda}}{f_p}\wedge \dots \wedge \frac{\debar |f_1|^{2\lambda}}{f_1} \right|_{\lambda=0},
\end{equation}
i.e., instead of letting the $\lambda_i$ go to zero successively, all the $\lambda_i$ are equal to a single
$\lambda$ that tends to $0$. In that case, Passare proves that this current coincides
with an average of limits along parabolic paths of the residue integral, as considered in \cite{PCrelle},
irrespectively of whether $f$ defines a complete intersection or not.

\begin{center}
---
\end{center}

\noindent The product $R^{k}\wedge \cdots \wedge R^1$ does in general not have any natural commutation properties. For instance,
$\debar(1/(zw))\wedge \debar (1/z)=0$ while $\debar(1/z)\wedge \debar(1/(zw))=\debar(1/z^2)\wedge \debar(1/w)$, where the 
last product simply is the tensor product. However, if the $f_j$ define a complete intersection, i.e., 
$\textrm{codim}\, \{f_1=\cdots=f_p=0\}= \sum_j\textrm{rank}\, E_j$, then it is known (see, e.g., \cite{MatsAArk}) 
that the product is commutative; the case when all the $E_j$ have rank $1$ is proved in \cite{CH}.
\begin{remark}
Recall that the currents $R^j$ take values in $\Lambda E_j$. The sum of the degree of $R^j$ in $\Lambda E_j$ and its 
form-degree is even. Therefore the product is naturally commutative. If the $E_j$ are trivial line bundles that we do not
make any distinction between, then the product is anti-commutative; this is the classical Coleff-Herrera setting.
\end{remark}

\begin{theorem}\label{simple-lambda}
Assume that the $f_j$ define a complete intersection. Then for every test form $\varphi$
\begin{equation*}
(\lambda_1,\cdots,\lambda_p) \mapsto
\int_Z R^{p,\lambda_p}\wedge \cdots \wedge R^{1,\lambda_1}\wedge \varphi
\end{equation*}
has an analytic continuation to a neighborhood of the origin in $\C^p$.
\end{theorem}

This result is a special case of our Theorem~\ref{lambda-main}, which generalizes 
\cite[Theorem 1]{HasamArkiv}. The case when $p=2$ and $\textrm{rank}\, E_j=1$ was proved by 
Berenstein-Yger (see, e.g., \cite{BGVY}). 
The following result is a special case of Theorem~\ref{epsilon-main}, which
generalizes \cite[Theorem 1]{JebHs}.

\begin{theorem}\label{jebhs+}
Assume that the $f_j$ define a complete intersection. Then for every test form $\varphi$
\begin{equation*}
(\epsilon_1,\ldots,\epsilon_p) \mapsto \int_Z R^{p,\epsilon_p}\wedge \cdots \wedge R^{1,\epsilon_1}\wedge \varphi
\end{equation*} 
is H\"{o}lder continuous on $[0,\infty)^p$.
\end{theorem}

For this result it is crucial that the $\chi$-functions used to regularize the $R^j$ are smooth. In fact,
Passare-Tsikh, \cite{PTex}, found a quite simple tuple $(f_1,f_2)$ defining a complete intersection in $\C^2$
and a test form $\varphi$ such that the classical Coleff-Herrera residue integral $I_{(f_1,f_2)}^{\varphi}(\epsilon)$
is discontinuous at $\epsilon=0$. Soon after Bj\"{o}rk found generic families of such examples, see, e.g., \cite{JebAbel}. 

\begin{center}
---
\end{center}

\noindent Let us give some background and motivation for the kind of products considered here.
Products of Cauchy-Fantappi\`{e}-Leray type currents were first studied by Wulcan, \cite{WArkiv}. Wulcan
defines the product as the value at $\lambda=0$ of the 
analytic continuation of $\lambda \mapsto R^{p,\lambda}\wedge \cdots \wedge R^{1,\lambda}$. In the non-complete intersection case 
Wulcan's product is different from our; in the case that all $E_j$ have rank $1$,
$R^{p,\lambda}\wedge \cdots \wedge R^{1,\lambda}|_{\lambda=0}$ coincides with Passare's product, \eqref{eqpdef}.
Passare-Wulcan products satisfy several 
natural computation rules and are quite useful but it has turned out that the recursive definition discussed above 
often is more natural. In particular, the St\"{u}ckrad-Vogel intersection algorithm in non-proper intersection theory
is conveniently expressed using recursively defined products, see \cite{ASWY}. 

In the complete intersection case there is no ambiguity, the Coleff-Herrera product is commutative and if 
$f=(f_1,\ldots,f_p)$ then $R^f$ equals $\wedge_j \debar(1/f_j)$, see \cite{PTY} and \cite{MatsAB}.
This indicates that the Coleff-Herrera product is the 
``correct'' current to associated to a complete intersection. The Coleff-Herrera product is the minimal current extension of 
Grothendieck's cohomological residue (see, e.g., \cite{Pdr} for definitions) in the sense that it annihilated by 
anti-holomorphic functions vanishing on its support. Moreover, if $f$ defines a complete intersection then 
the annihilator ideal of $R^f$ equals the ideal generated by $f$, see \cite{Pdr} and \cite{DS}.
This property is very useful and lies behind many applications, e.g., 
explicit division-interpolation formulas and Brian\c con-Skoda type results 
(\cite{MatsAM}, \cite{BGVY}), explicit versions
of the fundamental principle (\cite{BePa}), the $\debar$-equation on complex spaces (\cite{AS}, \cite{AS2}, \cite{HePo}), 
and explicit Green currents in arithmetic intersection theory (\cite{BYJAM}). 

\smallskip 

In Section~\ref{formulering}, we give the necessary background and the general formulations of our results.
Section \ref{bevis} contains the proof of Theorems \ref{sats1} and \ref{main}.
The proof of Theorems \ref{simple-lambda}, \ref{jebhs+}, \ref{epsilon-main} and \ref{lambda-main} is the content of 
Section~\ref{bevis2}; the crucial part is Lemma~\ref{divlemma} which enables us to effectively use the assumption 
about complete intersection.

\section{Formulation of the general results}\label{formulering}
Let $Z$ be a reduced complex space of pure dimension $n$. 
We say that $\varphi$ is a smooth $(p,q)$-form on $Z$ if $\varphi$ is smooth on $Z_{reg}$, and
in a neighborhood of any $p\in Z$, there is a smooth
$(p,q)$-form $\tilde{\varphi}$ in an ambient complex manifold such that the pullback of $\tilde{\varphi}$
to $Z_{reg}$ coincides with $\varphi\lvert_{Z_{reg}}$ close to $p$.
The $(p,q)$-test forms on $Z$, $\mathscr{D}_{p,q}(Z)$, are defined as the smooth compactly supported 
$(p,q)$-forms (with a suitable topology) and the space of $(p,q)$-currents on $Z$, $\mathscr{D}'_{p,q}(Z)$, is the
dual of $\mathscr{D}_{n-p,n-q}(Z)$. More concretely, if $i\colon Z\to \Omega \subset \C^N$ is an embedding and
$\mu$ is a $(p,q)$-current on $Z$ then $i_* \mu$ is an $(N-n+p,N-n+q)$-current in $\Omega$ that vanishes on
test forms $\xi$ such that $i^*\xi=0$ on $Z_{reg}$. Conversely, such a current in $\Omega$ defines a current
on $Z$. See, e.g., \cite{larkang} for a more thorough discussion.

Let $x$ be a complex coordinate on $\C$.
Recall that the principal value current $1/x^m$ can be computed as the value at $\lambda=0$
of the analytic continuation of $|x|^{2\lambda}/x^m$; the residue current $\debar (1/x^m)$ then is the value 
at $\lambda=0$ of $\debar |x|^{2\lambda}/x^m$. Since one can take tensor products of one-variable currents
it follows that
\begin{equation}\label{T}
T=\frac{1}{x_1^{\alpha_1}}\wedge \cdots \wedge \frac{1}{x_p^{\alpha_p}}\wedge 
\frac{\vartheta(x)}{x_{p+1}^{\alpha_{p+1}} \cdots x_n^{\alpha_n}}
\end{equation}
is a well defined current in $\C^n$; here $\alpha_1,\ldots,\alpha_p$ are positive integers, 
$\alpha_{p+1},\ldots,\alpha_n$ are non-negative integers, and
$\vartheta$ is a smooth compactly supported form. Such a current $T$ is called an {\em elementary}
{\em pseudomeromorphic} current.  
Following \cite{AWCrelle} we say that a current $\mu$ on $Z$ is {\em pseudomeromorphic}, $\mu \in \mathcal{PM}(Z)$,
if $\mu$ locally is a finite sum of push-forwards $\pi^1_* \cdots \pi^m_* \tau$ under maps
\begin{equation*}
X^m \stackrel{\pi^m}{\longrightarrow} \cdots
\stackrel{\pi^2}{\longrightarrow} X^1 \stackrel{\pi^1}{\longrightarrow} Z,
\end{equation*}
where each $\pi^j$ is either a modification or an open inclusion and $\tau$ is an elementary pseudomeromorphic 
current on $X^m$.
It follows that the class of pseudomeromorphic currents
is closed under $\debar$ and multiplication with smooth forms, and that the push-forward
of a pseudomeromorphic current by a modification is pseudomeromorphic.

\begin{lemma} \label{pmmultlemma}
    Let $f$ be a holomorphic function, and let $T \in \mathcal{PM}(Z)$.
    If $\tilde{f}$ is a holomorphic function such that $\{ \tilde{f} = 0 \} = \{ f = 0 \}$
    and $v$ is a smooth non-zero function, then $(|\tilde{f} v|^{2\lambda}/f) T$ and $(\debar |\tilde{f}v|^{2\lambda}/f)\wedge T$ have
    current-valued analytic continuations to $\lambda = 0$ and the values at $\lambda = 0$ are pseudomeromorphic
    and independent of the choices of $\tilde{f}$ and $v$.
    Moreover, if $\chi={\bf 1}_{[1,\infty)}$, or a smooth approximation thereof, then
    \begin{equation} \label{eqepsilonlambda}
        \left.\frac{|\tilde{f} v|^{2\lambda}}{f} T\right|_{\lambda = 0} = \lim_{\epsilon \to 0^+} \frac{\chi^\epsilon}{f} T
        \quad\text{and}\quad 
        \left.\frac{\debar |\tilde{f} v|^{2\lambda}}{f} \wedge T\right|_{\lambda = 0} = \lim_{\epsilon \to 0^+} \frac{\debar\chi^\epsilon}{f}\wedge T,
    \end{equation}
where $\chi^\epsilon = \chi(|\tilde{f}v|^2/\epsilon)$.
\end{lemma}

\begin{proof}
    The first part is essentially Proposition 2.1 in \cite{AWCrelle}, except that there, $Z$ is a complex manifold, $\tilde{f} = f$
    and $v \equiv 1$. However, with suitable resolutions of singularities, the proof in \cite{AWCrelle} goes through in the
    same way in our situation, as long as we observe that in $\C$
    \begin{equation*}
        \frac{|x^{\alpha'}v|^{2\lambda}}{x^\alpha} \frac{1}{x^\beta} \quad \text{and} \quad
        \frac{|x^{\alpha'}v|^{2\lambda}}{x^\alpha} \debar \frac{1}{x^\beta}
    \end{equation*}
    have analytic continuations to $\lambda = 0$, and the values at $\lambda = 0$ are $1/x^{\alpha + \beta}$
    and $0$ respectively, independently of $\alpha'$ and $v$, as long as $\alpha' > 0$ and $v \neq 0$
    (and similarly with $\debar |x^{\alpha'}v|^{2\lambda}/x^\alpha$).

    By Leibniz rule, it is enough to consider the first equality in \eqref{eqepsilonlambda}, since if we have proved the first equality, then
    \begin{align*}
        & \lim_{\epsilon \to 0} \frac{\debar \chi^\epsilon}{f}\wedge T = \lim_{\epsilon \to 0} \debar \left( \frac{\chi^\epsilon}{f} T  \right)
        - \frac{\chi^\epsilon}{f} \debar T \\
        &= \left.\left(\debar\left(\frac{|\tilde{f}v|^{2\lambda}}{f} T\right) -
        \frac{|\tilde{f}v|^{2\lambda}}{f} \debar T\right)\right|_{\lambda = 0}
        = \left.\frac{\debar |\tilde{f}v|^{2\lambda}}{f} \wedge T \right|_{\lambda = 0}.
    \end{align*}
    To prove the first equality in \eqref{eqepsilonlambda}, we observe first that in the same way as in the first part, we
    can assume that $f = x^{\gamma} u$ and $\tilde{f} = x^{\tilde{\gamma}} \tilde{u}$,
    where $u$ and $\tilde{u}$ are non-zero holomorphic functions.
    Since $T$ is a sum of push-forwards of elementary currents,
    we can assume that $T$ is of the form \eqref{T}. 
    Note that if $\supp \gamma \cap \supp \beta \neq \emptyset$, then
    $(|\tilde{f} v|^{2\lambda}/f) T = 0$ for $\mathfrak{Re}\, \lambda \gg 0$ and $(\chi(|\tilde{f}v|^2/\epsilon) /f) T = 0$
    for $\epsilon > 0$, since $\supp T \subseteq \{ x_i = 0, i \in \supp \beta \}$.
    Thus, we can assume that $\supp \gamma \cap \supp \beta = \emptyset$. By a smooth (but non-holomorphic) change of variables,
    as in Section~\ref{bevis} (equations \eqref{varbyte}), we can assume that $|\tilde{u} v|^2 \equiv 1$.
    Thus, since $(|x^{\tilde{\gamma}}|^{2\lambda} / x^\gamma) (1/x^\alpha)$, $(\chi(|x^{\tilde{\gamma}}|^2/\epsilon)/x^\gamma) (1/x^\alpha)$
    depend on variables disjoint from the ones that $\wedge_{\beta_i \neq 0} \debar (1/x_i^{\beta_i})$ depends on,
    it is enough to prove that
    \begin{equation*}
        \left.\frac{|x^{\tilde{\gamma}}|^{2\lambda}}{x^\gamma} \frac{1}{x^\alpha}\right|_{\lambda = 0} =
        \lim_{\epsilon \to 0} \frac{\chi(|x^{\tilde{\gamma}}|^2/\epsilon)}{x^\gamma} \frac{1}{x^\alpha},
    \end{equation*}
    which is Lemma 2 in \cite{JebHs}.
\end{proof}

\noindent Let $E_1,\ldots,E_q$ be holomorphic hermitian vector bundles 
over $Z$,
let $f_j$ be a holomorphic section of $E_j^*$, $j=1,\ldots,q$, and let $s_j$ be the section of $E_j$
with pointwise minimal norm such that $f_j \cdot s_j=|f_j|^2$. Outside $\{f_j=0\}$, define
\begin{equation*}
    u^j_k = \frac{s_j\wedge (\debar s_j)^{k-1}}{|f_j|^{2k}}.
\end{equation*}
It is easily seen that if $f_j = f_j^0 f_j'$, where $f_j^0$ is a holomorphic function and $f_j'$ is a non-vanishing section,
then $u^j_k = (1/f_j^0)^k (u')^j_k$, where $(u')^j_k$ is smooth across
$\{ f_j = 0 \}$.
We let
\begin{equation}\label{Udef}
U^j=\sum_{k=1}^{\infty} \left.|\tilde{f}_j|^{2\lambda} u^j_k \right|_{\lambda=0},
\end{equation}
where $\tilde{f}_j$ is any holomorphic section of $E_j^*$ such that $\{\tilde{f}_j=0\}=\{f_j=0\}$.
The existence of the analytic continuation is a local statement, so we can assume that 
$f_j = \sum f_{j,k} \mathfrak{e}_{j,k}^*$,
where $\mathfrak{e}_{j,k}^*$ is a local holomorphic frame for $E_j^*$. After principalization
we can assume that the ideal $\langle f_{j,1},\dots,f_{j,k_j} \rangle$ is generated by, e.g., $f_{j,0}$.
By the representation $u^j_k = (1/f_{j,0})^k (u')^j_k$,
the existence of the analytic continuation of $U^j$ in \eqref{Udef} then follows from Lemma \ref{pmmultlemma}.
Let $U^j_k$ denote the term of $U^j$ that takes values in $\Lambda^kE_j$; $U^j_k$ is thus
a $(0,k-1)$-current with values in $\Lambda^kE_j$. Let $\delta_{f_j}$ denote interior multiplication
with $f_j$ and put $\nabla_{f_j}=\delta_{f_j}-\debar$; it is not hard to verify that 
$\nabla_{f_j}U=1$ outside $f_j=0$. 
We define the Cauchy-Fantappi\`e-Leray type residue current, $R^j$, of $f_j$ by $R^j=1-\nabla_{f_j}U^j$.
One readily checks that 
\begin{eqnarray}\label{residydef}
R^j &=& R^j_0+\sum_{k=1}^{\infty}R^j_k \\
&=& (1-|\tilde{f}_j|^{2\lambda})|_{\lambda=0}+\sum_{k=1}^{\infty}
\left.\debar|\tilde{f}_j|^{2\lambda}\wedge\frac{s_j\wedge (\debar s_j)^{k-1}}{|f_j|^{2k}}\right|_{\lambda=0}, \nonumber
\end{eqnarray} 
where, as above, $\tilde{f}_j$ is a holomorphic section such that $\{\tilde{f}_j=0\}=\{f_j=0\}$.

\begin{remark}
Notice that if $E_j$ has rank $1$, then $U_j$ simply equals $1/f_j$ and 
$R^j=1-\nabla_{f_j} (1/f_j)=1-f_j\cdot (1/f_j)+\debar (1/f_j)=\debar (1/f_j)$.
\end{remark}

We now define a non-commutative calculus for the currents $U^i_k$ and $R^j_{\ell}$ recursively as follows.
\begin{definition}\label{proddef}
If $T$ is a product of some $U^i_k$ and $R^j_{\ell}$, then we define
\begin{eqnarray*}
& \bullet & U^j_k\wedge T=
\left.|\tilde{f}_j|^{2\lambda}\frac{s_j\wedge (\debar s_j)^{k-1}}{|f_j|^{2k}}\wedge T \right|_{\lambda=0} \\
& \bullet & \left.R^j_0\wedge T=(1-|\tilde{f}_j|^{2\lambda})T \right|_{\lambda=0} \\
& \bullet & R^j_k\wedge T=
\left.\debar|\tilde{f}_j|^{2\lambda}\wedge \frac{s_j\wedge (\debar s_j)^{k-1}}{|f_j|^{2k}}\wedge T \right|_{\lambda=0},
\end{eqnarray*}
where $\tilde{f}_j$ is any holomorphic section of $E^*_j$ with $\{\tilde{f}_j=0\}=\{f_j=0\}$.
\end{definition}

Notice that after principalization the pull-back of $u_k^j$ is semi\hyp{}meromorphic; in particular 
$U^j$ and $R^j$ are pseudomeromorphic.  
Thus, by Lemma~\ref{pmmultlemma}, the analytic continuations
of Definition~\ref{proddef} exist and the values at $\lambda=0$ are pseudomeromorphic as well.

\begin{remark}
Under assumptions about complete intersection, these products have the suggestive
commutation properties, e.g., if $\textrm{codim}\, \{f_i=f_j=0\}= \rank E_i + \rank E_j$,
then $R^i_k\wedge R^j_{\ell}=R^j_{\ell}\wedge R^i_k$, $R^i_k\wedge U^j_{\ell}=U^j_{\ell}\wedge R^i_k$, and
$U^i_k\wedge U^j_{\ell}=-U^j_{\ell}\wedge U^i_k$, (see, e.g., \cite{MatsAArk}). 
In general, there are no simple relations.
However, products involving only $U$:s are always anti-commutative.
\end{remark}

Now, consider collections $R=\{R^1_{k_1},\ldots, R^p_{k_p}\}$ and 
$U=\{U^{p+1}_{k_{p+1}},\ldots,U^q_{k_q}\}$
and put
$(P_1,\ldots,P_q)=(R^1_{k_1},\ldots,R^p_{k_p},U^{p+1}_{k_{p+1}},\ldots,U^q_{k_q})$. For a permutation $\nu$ of 
$\{1,\ldots,q\}$ we define
\begin{equation}\label{URdef}
(UR)^{\nu}=P_{\nu(q)}\wedge \cdots \wedge P_{\nu(1)}.
\end{equation}

From \eqref{Udef} and \eqref{residydef} we get natural $\lambda$-regularizations, $P^{\lambda}_j$, 
of $P_j$ and from Definition~\ref{proddef} we have
$(UR)^{\nu}= P^{\lambda_q}_{\nu(q)}\wedge \cdots \wedge P^{\lambda_1}_{\nu(1)} |_{\lambda_1=0} \cdots|_{\lambda_q=0}$,
i.e., we set successively $\lambda_1=0$, then $\lambda_2 = 0$ and so on.
The following result is proved in \cite{ASWY}.

\begin{theorem}\label{aswy}
Let $a_1>\cdots >a_q>0$ be integers and $\lambda$ a complex variable. Then
\begin{equation*}
\lambda\mapsto P^{\lambda^{a_q}}_{\nu(q)}\wedge \cdots \wedge P^{\lambda^{a_1}}_{\nu(1)}
\end{equation*}
has a current-valued analytic continuation to a neighborhood of the half-axis $[0,\infty)\subset \C$
and the value at $\lambda=0$ equals $(UR)^{\nu}$.
\end{theorem}

The recursively defined product $(UR)^{\nu}$ can thus be obtained as the value at zero of a one-variable $\zeta$-type
function. From an algebraic point of view, this is desirable since one can derive functional equations
and use Bernstein-Sato theory to study $(UR)^{\nu}$. 

\smallskip

There are also more concrete and explicit regularizations of the currents
$U^i_k$ and $R^j_{\ell}$ inspired by \cite{CH} and \cite{PCrelle}. Let $\chi={\bf 1}_{[1,\infty)}$, or 
a smooth approximation thereof that is $0$ close to $0$ and $1$ close to $\infty$. 
It follows from \cite{hasamJFA}, or after principalization from Lemma \ref{pmmultlemma}, that 
\begin{equation}\label{Uepsilon}
U^j_k=\lim_{\epsilon\to 0^+}\chi(|\tilde{f}_j|^2/\epsilon) \frac{s_j\wedge (\debar s_j)^{k-1}}{|f_j|^{2k}}.
\end{equation}
\begin{equation}\label{Repsilon}
R^j_k=\lim_{\epsilon\to 0^+}\debar\chi(|\tilde{f}_j|^2/\epsilon)\wedge 
\frac{s_j\wedge (\debar s_j)^{k-1}}{|f_j|^{2k}},\,\,
k>0,
\end{equation}
and similarly for $k=0$; as usual, $\{\tilde{f}_j=0\}=\{f_j=0\}$. 
Of course, the limits are in the current sense and if $\chi={\bf 1}_{[1,\infty)}$,
then $\epsilon$ is supposed to be a regular value for $|f_j|^2$ and $\debar\chi(|f_j|^2/\epsilon)$ is to be 
interpreted as integration over the manifold $|f_j|^2=\epsilon$. We denote the regularizations 
given by \eqref{Uepsilon} and \eqref{Repsilon} by $P_j^{\epsilon}$.

\begin{theorem}\label{main}
Let $R=\{R^1_{k_1},\ldots,R^p_{k_p}\}$ and $U=\{U^{p+1}_{k_{p+1}},\ldots,U^q_{k_q}\}$
be collections of currents defined in \eqref{Udef} and \eqref{residydef}.
Let $\nu$ be a permutation of $\{1,\ldots,q\}$ and let $(UR)^{\nu}$ be the product defined in \eqref{URdef}.
Then
\begin{equation*}
(UR)^{\nu}=\lim_{\epsilon_1 \ll \cdots \ll \epsilon_q \to 0}P_{\nu(q)}^{\epsilon_q}\wedge \cdots \wedge P_{\nu(1)}^{\epsilon_1},
\end{equation*}
where, as above, $(P_1,\ldots,P_q)=(R^1_{k_1},\ldots,R^p_{k_p},U^{p+1}_{k_{p+1}},\ldots,U^q_{k_q})$; see 
Definition~\ref{limitdef} for the meaning of the limit. If $\chi={\bf 1}_{[1,\infty)}$, 
we require that $\epsilon\to 0$ along an admissible path in the sense of Coleff-Herrera.
\end{theorem}

Thus $(UR)^{\nu}$ can be computed as the weak limit of an explicit smooth form and moreover, Definition~\ref{proddef} give the 
Coleff-Herrera product (in case the bundles $E_j$ have rank $1$).

\begin{remark}
It might be more natural to consider products of whole Cauchy-Fantappi\`e-Leray type currents,
$U^j$ and $R^j$, as in \eqref{Udef} and \eqref{residydef}, and not just products of their
components $U^j_k$ and $R^j_k$, cf., for example \cite{ASWY}.
However, since such a product is a sum of products of their components, it follows readily
that Theorem~\ref{main} holds also for products of whole Cauchy-Fantappi\`e-Leray type
currents.
\end{remark}

\subsection{The complete intersection case}

Assume that $f_1,\ldots,f_q$ define a complete intersection, i.e., that 
$\textrm{codim}\, \{f_1=\cdots =f_q=0\}=\rank E_1+\cdots + \rank E_q$. Then we know that 
the calculus defined in Definition~\ref{proddef} satisfies the suggestive commutation properties, but 
we have in fact the following much stronger results.

\begin{theorem}\label{epsilon-main}
Assume that $f_1,\ldots,f_q$ define a complete intersection on $Z$, let
$(P_1,\ldots,P_q)=(R^1_{k_1},\ldots,R^p_{k_p},U^{p+1}_{k_{p+1}},\ldots,U^q_{k_q})$, and let
$P^{\epsilon_j}_{j}$ be an $\epsilon$-regularization of $P_j$ defined by \eqref{Uepsilon} and \eqref{Repsilon}
with smooth $\chi$-functions. Then we have
\begin{equation*}
\left| \int_Z P^{\epsilon_1}_1\wedge \cdots \wedge P_q^{\epsilon_q}\wedge \varphi -
P_1\wedge \cdots \wedge P_q .\, \varphi \right| \leq C \|\varphi\|_{C^M} (\epsilon_1^{\omega} + \dots + \epsilon_q^\omega),
\end{equation*}
where $M$ and $\omega$ only depend on $f_1,\ldots, f_q$, $Z$, and $\supp \varphi$ while
$C$ also depends on the $C^M$-norm of the $\chi$-functions.
\end{theorem}

\begin{theorem}\label{lambda-main}
Assume that $f_1,\ldots,f_q$ define a complete intersection on $Z$, let
$(P_1,\ldots,P_q)=(R^1_{k_1},\ldots,R^p_{k_p},U^{p+1}_{k_{p+1}},\ldots,U^q_{k_q})$, and let
$P^{\lambda_j}_{j}$ be the $\lambda$-regularization of $P_j$ given by \eqref{Udef} and \eqref{residydef}.
Then the current valued function
\begin{equation*}
\lambda \mapsto P_1^{\lambda_1}\wedge \cdots \wedge P_q^{\lambda_q},
\end{equation*}
a priori defined for $\mathfrak{Re}\, \lambda_j \gg 0$, has an analytic continuation
to a neighborhood of the half-space $\cap_1^q \{\mathfrak{Re}\, \lambda_j \geq 0\}$.
\end{theorem}

\begin{remark}
In case the $E_j$ are trivial with trivial metrics, Theorems \ref{epsilon-main} and \ref{lambda-main} follow
quite easily from, respectively, \cite[Theorem 1]{JebHs} and \cite[Theorem 1]{HasamArkiv} by taking averages.
As an illustration, let 
$\varepsilon_1,\ldots,\varepsilon_r$ be a nonsense basis and let $f_1,\ldots,f_r$ be holomorphic functions. 
Then we can write $s=\bar{f}\cdot \varepsilon$ and so 
$u_k=(\bar{f}\cdot \varepsilon)\wedge (d\bar{f}\cdot \varepsilon)^{k-1}/|f|^{2k}$.
A standard computation shows that 
\begin{equation*}
\int_{\alpha \in \CP^{r-1}}\frac{|\alpha \cdot f|^{2\lambda}\alpha \cdot \varepsilon}{
(\alpha\cdot f)|\alpha|^{2\lambda}}dV
=A(\lambda)|f|^{2\lambda}\frac{\bar{f}\cdot \varepsilon}{|f|^2},
\end{equation*}
where $dV$ is the (normalized) Fubini-Study volume form and $A$ is holomorphic with $A(0)=1$. It follows that
\begin{equation*}
\int_{\alpha_1,\ldots,\alpha_k\in \CP^{r-1}}\bigwedge_1^k\frac{\debar |\alpha_j\cdot f|^{2\lambda}}{\alpha_j\cdot f}
\wedge \frac{\alpha_j \cdot \varepsilon}{|\alpha_j|^{2\lambda}}dV(\alpha_j)=
A(\lambda)^k\debar (|f|^{2k\lambda}u_k).
\end{equation*}
Elaborating this formula and using \cite[Theorem 1]{HasamArkiv} one can show Theorem~\ref{lambda-main} in the case
of trivial $E_j$ with trivial metrics. The general case can probably also be handled in a similar manner but the 
computations become more involved and we prefer to give direct proofs.
\end{remark}

\section{Proof of Theorem \ref{main}}\label{bevis}

The structure of this proof is rather similar to the structure of the proof
of Proposition~5.4 in \cite{ASWY}.

We start by making a Hironaka resolution of singularities, \cite{Hiro}, of $Z$ such that the pre-image of
$\cup_j\{f_j=0\}$ has normal crossings. We then make further toric resolutions (e.g., as in \cite{PTY}) 
such that, in local charts, the pullback of each $f_i$ is a monomial, $x^{\alpha_i}$, times a 
non-vanishing holomorphic tuple. One checks that the pullback of $P_j^{\epsilon}$ is of one of the following
forms: 
\begin{equation*}
\frac{\chi(|x^{\tilde{\alpha}}|^2\xi/\epsilon)}{x^{\alpha}}\, \vartheta, \quad
1-\chi(|x^{\tilde{\alpha}}|^2\xi/\epsilon),\quad
\frac{\debar \chi(|x^{\tilde{\alpha}}|^2\xi/\epsilon)}{x^{\alpha}}\wedge \vartheta,
\end{equation*}
where $\xi$ is smooth and positive, $\supp \tilde{\alpha}= \supp \alpha$,
and $\vartheta$ is a smooth bundle valued form; by localizing on the blow-up we may also suppose that 
$\vartheta$ has as small support as we wish. If the $\chi$-functions are smooth, the following special
case of Theorem \ref{main} now immediately follows from Lemma \ref{pmmultlemma}:
\begin{equation}\label{eq1}
(UR)^{\nu}=\lim_{\epsilon_q\to 0}\cdots \lim_{\epsilon_1\to 0}P^{\epsilon_q}_{\nu(q)}\wedge \cdots \wedge 
P^{\epsilon_1}_{\nu(1)}.
\end{equation} 

\bigskip

For smooth $\chi$-functions we put
\begin{equation*}
\mathcal{I}(\epsilon)=
\int \frac{\debar \chi_1^{\epsilon}\wedge \cdots \wedge \debar \chi_p^{\epsilon}
\chi_{p+1}^{\epsilon}\cdots \chi_q^{\epsilon}}{x^{\alpha_1+\cdots +\alpha_p+\cdots +\alpha_{q'}}}\wedge \varphi,
\end{equation*}
where $q'\leq q$, $\varphi$ is a smooth $(n,n-p)$-form with support close to the origin, and 
$\chi_j^{\epsilon}=\chi (|x^{\tilde{\alpha_j}}|^2\xi_j/\epsilon_j)$ for smooth positive $\xi_j$.
We note that we may replace the $\debar$ in $\mathcal{I}(\epsilon)$ by $d$ for bidegree reasons.
In case $\chi={\bf 1}_{[1,\infty)}$ we denote the corresponding integral by $I(\epsilon)$. 
We also put $\mathcal{I}^{\nu}(\epsilon_1,\ldots,\epsilon_q)=\mathcal{I}(\epsilon_{\nu(1)},\ldots,\epsilon_{\nu(q)})$
and similarly for $I^{\nu}$. In view of \eqref{eq1}, the special case of Theorem \ref{main}
when the $\chi$-functions are smooth will be proved if we can show that
\begin{equation}\label{eq2}
\lim_{\epsilon_1 \ll \cdots \ll \epsilon_q \to 0}\mathcal{I}^{\nu}(\epsilon)
\end{equation}
exists. 
The case with $\chi={\bf 1}_{[1,\infty)}$ will then follow if we can show
\begin{equation}\label{eq3}
\lim_{\delta\to 0} (\mathcal{I}^{\nu}(\epsilon(\delta))-I^{\nu}(\epsilon(\delta)))=0,
\end{equation}
where $\delta \mapsto \epsilon(\delta)$ is any admissible path.

For notational convenience, we will consider $\mathcal{I}^\nu(\epsilon)$ (unless otherwise stated), but our 
arguments apply just as well to $I^\nu(\epsilon)$ until we arrive at the integral \eqref{Iepsilon}.

Denote by $\tilde{A}$ the $q\times n$-matrix with rows $\tilde{\alpha}_i$. 
We will first show that we can assume that $\tilde{A}$ has full rank. The idea is the same as in \cite{CH} and \cite{PCrelle},
however because of the paths along which our limits are taken, we have to modify the argument slightly.
The following lemma follows from the proof of Lemma III.12.1 in \cite{TsikhBook}.
\begin{lemma}\label{ranklemma}
Assume that $\alpha$ is a $q \times n$-matrix with rows $\alpha_i$ such that there exists $(v_1,\dots,v_q) \neq 0$ with $\sum v_i\alpha_i = 0$.
Let $j = \min \{ i; v_i \neq 0 \}$. Then there exist constants $C,c > 0$ such that if $\epsilon_{j} < C(\epsilon_{j+1}\dots\epsilon_q)^c$,
then $\chi(|x^{\alpha_j}|^2\xi_j/\epsilon_j) \equiv 1$ and $\debar\chi(|x^{\alpha_j}|^2\xi_j/\epsilon_j) \equiv 0$ for
all $x \in \Delta \cap \{ |x^{\alpha_i}|^2 \geq C_i\epsilon_i, i = j+1,\dots,q \}$, where $\Delta$ is the unit polydisc.
\end{lemma}
Assume that $\tilde{A}$ does not have full rank, and let $v$ be a column vector such that $v^t \tilde{A} = 0$. Since
$(\epsilon_1,\dots,\epsilon_q)$ is replaced by $(\epsilon_{\nu(1)},\dots,\epsilon_{\nu(q)})$ in $\mathcal{I}^\nu(\epsilon)$,
we choose instead $j_0$ such that $\nu(j_0) \leq \nu(i)$ for all $i$ such that $v_i \neq 0$.
If $j_0 \leq p$, we let $\widetilde{\mathcal{I}}^\nu(\epsilon) = 0$, and if $j_0 \geq p+1$, we let $\widetilde{\mathcal{I}}^\nu(\epsilon)$
be $\mathcal{I}^\nu(\epsilon)$ but with $\chi_{j_0}^\epsilon$ replaced by $1$.
If $\epsilon = \epsilon(\delta)$ is such that $\epsilon_{\nu(j_0)} > 0$, then $\mathcal{I}^\nu(\epsilon)$ is a current
acting on a test form with support on a set of the form
\begin{equation*}
    \Delta \cap \{ |x^{\alpha_i}|^2 \geq C_i\epsilon_{\nu(i)}; \text{for all $i$ such that } \nu(i) \geq \nu(j_0) \}.
\end{equation*}
In particular, if $\epsilon_{\nu(j_0)}(\delta)$ is sufficiently small compared to $(\epsilon_{\nu(j_0)+1}(\delta),\dots,$
$\epsilon_q(\delta))$,
then by Lemma \ref{ranklemma}, if $j_0 \leq p$, the factor $\debar\chi_{j_0}^\epsilon$ is identically $0$, and if $j_0 \geq p+1$,
the factor $\chi_{j_0}^\epsilon$ is identically $1$
and thus is equal to $\widetilde{\mathcal{I}}^\nu(\epsilon)$ for such $\epsilon$.
Similarly, if $\epsilon_{\nu(j_0)} = 0$, we have that $\mathcal{I}^\nu(\epsilon)$ is defined as a limit along $\epsilon_{\nu(j_0)} \to 0$,
with $\epsilon_{\nu(j_0)+1},\dots,\epsilon_q$ fixed and in the limit we get again that for sufficiently small $\epsilon_{\nu(j_0)}$,
we can replace $\mathcal{I}^\nu(\epsilon)$ by $\widetilde{\mathcal{I}}^\nu(\epsilon)$.
Thus we have
\begin{equation*}
\lim_{\epsilon_1 \ll \dots \ll \epsilon_q \to 0} \mathcal{I}^\nu(\epsilon) = 
\lim_{\epsilon_1 \ll \dots \ll \epsilon_q \to 0} \widetilde{\mathcal{I}}^\nu(\epsilon),
\end{equation*}
and we have reduced to the case that $\tilde{A}$ is a $(q-1)\times n$-matrix of the same rank. We continue
this procedure until $\tilde{A}$ has full rank.

\smallskip

By re-numbering the coordinates, we may suppose that the minor
$A=(\tilde{\alpha}_{ij})_{1\leq i,j\leq q}$ of $\tilde{A}$ is invertible and we put $A^{-1}=B=(b_{ij})$. 
We now use complex notation to 
make a non-holomorphic, but smooth change of variables:
\begin{equation}\label{varbyte}
y_1=x_1\,\xi^{b_1/2},\ldots, y_q=x_q\,\xi^{b_q/2}, y_{q+1}=x_{q+1},\ldots, y_n=x_n,
\end{equation}
\begin{equation*}
\hspace{.7cm} \bar{y}_1=\bar{x}_1\,\xi^{b_1/2},\ldots, \bar{y}_q=\bar{x}_q\,\xi^{b_q/2}, 
\bar{y}_{q+1}=\bar{x}_{q+1},\ldots, \bar{y}_n=\bar{x}_n,
\end{equation*}
where $\xi^{b_i/2}=\xi_1^{b_{i1}/2}\cdots \xi_q^{b_{iq}/2}$. One easily checks that 
$dy\wedge d\bar{y}=\xi^{b_1}\cdots \xi^{b_q}\,$ $dx\wedge d\bar{x}+ O(|x|)$, so \eqref{varbyte} defines a
smooth change of variables between neighborhoods of the origin. 
A simple linear algebra computation then shows that
$|x^{\tilde{\alpha_i}}|^2\xi_i=|y^{\tilde{\alpha_i}}|^2$. Of course, this change of variables does not preserve 
bidegrees so $\varphi(y)$ is merely a smooth compactly supported $(2n-p)$-form.
We thus have
\begin{equation}\label{I1(y)}
\mathcal{I}^\nu(\epsilon)=
\int_{\Delta} \frac{d\chi_1^{\epsilon}\wedge \cdots \wedge d \chi_p^{\epsilon}
\chi_{p+1}^{\epsilon}\cdots \chi_q^{\epsilon}}{y^{\alpha_1+\cdots +\alpha_p+\cdots +\alpha_{q'}}}\wedge \varphi'(y),
\end{equation}
where $\chi_j^{\epsilon}=\chi(|y^{\tilde{\alpha}_j}|^2/\epsilon_{\nu(j)})$ and 
$\varphi'(y)=\sum_{|I|+|J|=2n-p}\psi_{IJ}dy_I\wedge d\bar{y}_J$. By linearity we may assume that the sum only consists
of one term $\varphi'(y)=\psi dy_K\wedge d\bar{y}_L$, and by scaling, we may assume that 
$\supp \psi \subseteq \Delta$, $\Delta$ being the unit polydisc.
By Lemma 2.4 in \cite{CH}, we can write the function $\psi$ as
\begin{equation}\label{taylor}
\psi(y)=\sum_{I+J<\sum_1^{q'}\alpha_j-{\bf 1}}\psi_{IJ}\, y^I\bar{y}^J +
\sum_{I+J=\sum_1^{q'}\alpha_j-{\bf 1}}\psi_{IJ}\, y^I\bar{y}^J,
\end{equation}
where $a<b$ for tuples $a$ and $b$ means that $a_i<b_i$ for all $i$. In the decomposition \eqref{taylor}
each of the smooth functions $\psi_{IJ}$ in the first sum on the left-hand side 
is independent of some variable. We now show that this implies that the first sum on the left-hand side of
\eqref{taylor} does not contribute to the integral \eqref{I1(y)}. In case $\varphi'(y)$ has bidegree $(n,n-p)$
this is a well-known fact but we must show it for an arbitrary $(2n-p)$-form.

We change to polar coordinates:
\begin{equation*}
dy_K\wedge d\bar{y}_L=d(r_{K_1}e^{i\theta_{K_1}})\wedge \cdots \wedge d(r_{L_1}e^{-i\theta_{L_1}})\wedge \cdots
\end{equation*}
Since $\chi_j^{\epsilon}$ in \eqref{I1(y)} is independent of $\theta$, it follows that we must have full
degree $=n$ in $d\theta$. The only terms in the expansion of $dy_K\wedge d\bar{y}_L$ above that will contribute
to \eqref{I1(y)} are therefore of the form
\begin{equation*}
c\, r_1\cdots r_ne^{i\theta \cdot \gamma}\, dr_M\wedge d\theta,
\end{equation*}
where $|M|=n-p$, $c$ is a constant, and $\gamma$ is a multiindex with entries equal to $1$, $-1$, or $0$.
Substituting this and a term $\psi_{IJ}y^I\bar{y}^J=\psi_{IJ}r^{I+J}e^{i\theta \cdot (I-J)}$ from \eqref{taylor}
into \eqref{I1(y)} gives rise to an ``inner'' $\theta$-integral (by Fubini's theorem):
\begin{equation*}
\mathscr{J}_{IJ}(r)=
\int_{\theta \in [0,2\pi)^n}\psi_{IJ}(r,\theta) \ e^{i\theta \cdot (I-J-\sum_1^{q'}\alpha_j+\gamma)}\, d\theta.
\end{equation*}
If $I+J<\sum_1^{q'}\alpha_j-{\bf 1}$, then $I-J-\sum_1^{q'}\alpha_j+\gamma <0$ and $\psi_{IJ}$ is independent 
of some $y_j=r_je^{i\theta_j}$. Integrating over $\theta_j\in [0,2\pi)$ thus yields $\mathscr{J}_{IJ}=0$
if $I+J<\sum_1^{q'}\alpha_j-{\bf 1}$. If instead $I+J=\sum_1^{q'}\alpha_j-{\bf 1}$, then $\mathscr{J}_{IJ}(r)$
is smooth on $[0,\infty)^n$.

Summing up, we see that we can write \eqref{I1(y)} as
\begin{equation} \label{Iepsilon}
\mathcal{I}^\nu(\epsilon)=
\int_{r\in (0,1)^n} d\chi_1^{\epsilon}\wedge \cdots \wedge d \chi_p^{\epsilon}
\chi_{p+1}^{\epsilon}\cdots \chi_q^{\epsilon}\, \mathscr{J}(r)\, dr_M,
\end{equation}
where $\chi_j^{\epsilon}=\chi(r^{2\alpha_j}/\epsilon_{\nu(j)})$, $\mathscr{J}$ is smooth, and $|M|=n-p$.

After these reductions, the integral \eqref{Iepsilon} we arrive at is the same as equation (16) in \cite{PCrelle},
and we will use the fact proven there, that $\lim_{\delta \to 0} \mathcal{I}^\nu(\epsilon(\delta))$ exists
along any admissible path $\epsilon(\delta)$, and is well-defined independently of the choice of
admissible path. (This is not exactly what is proven there, but the fact that if $b \in \Q^p$,
then $\lim_{\delta \to 0} \epsilon(\delta)^b$ is either $0$ or $\infty$ independently of the
admissible path chosen is the only addition we need to make for the argument to go through in our case.)
Using this, if we let $\epsilon(\delta)$ be any admissible path, we will show by induction over $q$ that
\begin{equation*}
    \lim_{\epsilon_1 \ll \dots \ll \epsilon_q \to 0 } \mathcal{I}^\nu(\epsilon) = \lim_{\delta \to 0} \mathcal{I}^\nu(\epsilon(\delta)).
\end{equation*}
For $q = 1$ this is trivially true, so we assume $q > 1$. Let $\epsilon^k$ be any sequence satisfying the conditions
in Definition~\ref{limitdef}.
Consider a fixed $k$, and let $m$ be such that $\epsilon^k = (0,\dots,0,\epsilon^k_{m+1},\dots,\epsilon^k_{q})$ with $\epsilon_{m+1}^k > 0$.
Let $I_1 = \nu^{-1}(\{1,\dots,m\})\cap\{1,\dots,p\}$ and $I_2 = \nu^{-1}(\{1,\dots,m\})\cap\{p+1,\dots,q\}$.
We consider $\epsilon^k_{m+1},\dots,\epsilon^k_{q}$ fixed in $\mathcal{I}^\nu(\epsilon)$, and define
\begin{equation*}
    \mathcal{I}_k(\epsilon_1,\dots,\epsilon_m) = \int_{[0,1]^n} \bigwedge_{i \in I_1} d\chi(r^{\alpha_i}/\epsilon_{\nu(i)})
    \prod_{i \in I_2} \chi(r^{\alpha_i}/\epsilon_{\nu(i)})\mathscr{J}_k(r)dr_M,
\end{equation*}
originally defined on $(0,\infty)^p$, but extended according to Definition \ref{limitdef}, where
\begin{equation*}
    \mathscr{J}_k(r) = \pm \bigwedge_{i \in \{ 1,\dots,p \} \setminus I_1} d\chi(r^{\alpha_i}/\epsilon^k_{\nu(i)})
    \prod_{i \in \{ p+1,\dots,q\} \setminus I_2} \chi(r^{\alpha_i}/\epsilon^k_{\nu(i)})\mathscr{J}(r)
\end{equation*}
(where the sign is chosen such that $\mathcal{I}_k(0) = \mathcal{I}^\nu(\epsilon^k)$).
Since $m < q$ and $\mathscr{J}_k$ is smooth, we have by induction that
\begin{equation*}
    \mathcal{I}_k(0) = \lim_{\epsilon_m \to 0} \dots \lim_{\epsilon_1 \to 0} \mathcal{I}_k(\epsilon_1,\dots,\epsilon_m) =
    \lim_{\delta \to 0} \mathcal{I}_k(\epsilon'(\delta)),
\end{equation*}
where $\epsilon'(\delta)$ is any admissible path, and the first equality follows by definition of $\mathcal{I}_k(0)$.
We fix an admissible path $\epsilon'(\delta)$. For each $k$ we can choose $\delta_k$ such that
if $\epsilon^{k'} = (\epsilon'_1(\delta_k),\dots,\epsilon'_m(\delta_k))$, then
$\lim_{k \to \infty} (\mathcal{I}_k(\epsilon^{k'}) - \mathcal{I}_k(0)) = 0$ and
if $\tilde{\epsilon}^k = (\epsilon^{k'},\epsilon^k_{m+1},\dots,\epsilon^k_{q})$, then $\tilde{\epsilon}^k$ forms
a subsequence of an admissible path.
Since $\mathcal{I}_k(0) = \mathcal{I}^\nu(\epsilon^k)$, and $\mathcal{I}_k(\epsilon^{k'}) = \mathcal{I}^\nu(\tilde{\epsilon}^k)$, we thus have
\begin{equation*}
    \lim_{k \to \infty} \mathcal{I}^\nu(\epsilon^k) = \lim_{k \to \infty} \mathcal{I}^\nu(\tilde{\epsilon}^k) =
    \lim_{\delta \to 0} \mathcal{I}^\nu(\epsilon(\delta))
\end{equation*}
where the second equality follows from the existence and uniqueness of $\mathcal{I}^\nu(\epsilon(\delta))$
along any admissible path.
Hence we have shown that the limit in \eqref{eq2} exists and is well-defined.

Finally, if we start from \eqref{Iepsilon}, as (23) in \cite{PCrelle} shows, either
\begin{equation*}
    \lim_{\epsilon_1 \ll \dots \ll \epsilon_q \to 0} \mathcal{I}^\nu(\epsilon) = \pm \int_{r_M \in (0,1)^{n-p}} 
    \mathscr{J}(0,r_M) dr_M,
\end{equation*}
or the limit is $0$, depending only on $\alpha$. If we consider $I^{\nu}(\epsilon)$ instead, we get the same limit,
see \cite[p. 79--80]{TsikhBook}, and \eqref{eq3} follows.

\section{Proof of Theorems \ref{epsilon-main} and \ref{lambda-main}}\label{bevis2}
As in \cite{HasamArkiv} and \cite{JebHs} the key-step of the proof is a Whitney type division lemma, 
Lemma~\ref{divlemma} below.
Recall that 
\begin{equation*}
(P_1,\ldots,P_q)=(R^1_{k_1},\ldots,R^p_{k_p},U^{p+1}_{k_{p+1}},\ldots,U^q_{k_q})
\end{equation*}
and that 
$P_j^{\epsilon_j}$ and $P_j^{\lambda_j}$ are the $\epsilon$-regularizations with smooth $\chi$ 
(given by \eqref{Uepsilon}, \eqref{Repsilon}) 
and the $\lambda$-regularizations (cf., \eqref{Udef}, \eqref{residydef}) respectively of $P_j$.  
We will consider the following two integrals:
\begin{equation*}
\mathcal{I}(\epsilon)=\int_Z
P_1^{\epsilon_1}\wedge \cdots \wedge P_q^{\epsilon_q}\wedge \varphi
\end{equation*}
\begin{equation*}
\Gamma(\lambda) =\int_Z P_1^{\lambda_1}\wedge \cdots \wedge P_q^{\lambda_q}\wedge \varphi,  
\end{equation*}
where $\varphi$
is a test form on $Z$, supported close to a point in $\{f_1=\cdots=f_q=0\}$, 
of bidegree $(n,n-k_1-\cdots- k_q +q-p)$ with values in $\Lambda (E_1^*\oplus \cdots \oplus E_q^*)$. 
In the arguments below, we will assume for notational convenience that $\tilde{f}_j=f_j$ 
(cf., e.g., \eqref{Udef}); the modifications to the general case are straightforward.

The main parts of the proofs of Theorems \ref{epsilon-main} and \ref{lambda-main}
are contained in the following propositions.

\begin{proposition}\label{epsilonpropp}
Assume that $f_1,\ldots,f_q$ define a complete intersection. For $p<s\leq q$ we have
\begin{equation*}
\big| \mathcal{I}(\epsilon)-\mathcal{I}(\epsilon_1,\ldots,\epsilon_{s-1},0,\ldots,0)\big|\leq
C\|\varphi\|_M (\epsilon_{s}^{\omega}+\cdots +\epsilon_q^{\omega}).
\end{equation*}
Note that $\mathcal{I}(\epsilon_1,\ldots,\epsilon_{s-1},0,\ldots,0)$ is well-defined; it is the action 
of $U^s_{k_s}\wedge \cdots \wedge U^q_{k_q}$ on a smooth form. 
\end{proposition}

\begin{proposition}\label{lambdapropp}
Assume that $f_1,\ldots,f_q$ define a complete intersection.
Then $\Gamma(\lambda)$ has a meromorphic continuation to all of $\C^q$ and its only possible poles in a neighborhood 
of $\cap_1^q\{\mathfrak{Re}\, \lambda_j\geq 0\}$ are along hyperplanes of the form
$\sum_{j=1}^p\lambda_j \alpha_j=0$,
where $\alpha_j\in \mathbb{N}$ and at least two $\alpha_j$ are positive.
In particular, for $p=1$, $\Gamma(\lambda)$ is analytic in 
a neighborhood of $\cap_{1}^q\{\mathfrak{Re}\, \lambda_j\geq 0\}$.
\end{proposition} 

Using that
\begin{equation}\label{nablalikhet}
\debar |f_j|^{2\lambda}\wedge u^j_k=\debar (|f_j|^{2\lambda}u^j_k)-f_j\cdot (|f_j|^{2\lambda}u^j_{k+1}),
\end{equation}
the proof of Theorem \ref{lambda-main} follows from Proposition \ref{lambdapropp} in a similar way as
Theorem 1 in \cite{HasamArkiv} follows from Proposition 4 in \cite{HasamArkiv}.

\bigskip

We indicate one way Proposition~\ref{epsilonpropp} can be used to prove Theorem~\ref{epsilon-main}.
To simplify notation somewhat, we let $R^j$ denote any $R^j_k$ and $R^j_{\epsilon}$ denotes a smooth 
$\epsilon$-regularization of $R^j$; $U^j$ and $U^j_{\epsilon}$ are defined similarly.
 The uniformity in the estimate of Proposition \ref{epsilonpropp} implies that
we have estimates of the form
\begin{equation}\label{foljd}
\left|\bigwedge_1^m R^j_{\epsilon}\wedge \bigwedge_{m+1}^p R^j \wedge \bigwedge_{p+1}^q U^j_{\epsilon}-
\bigwedge_1^m R^j_{\epsilon}\wedge \bigwedge_{m+1}^p R^j \wedge \bigwedge_{p+1}^q U^j\right|\lesssim 
(\epsilon_{p+1}^{\omega}+\cdots +\epsilon_q^{\omega}),
\end{equation}
where, e.g., $R^{m+1}\wedge \cdots \wedge R^p$ a priori is defined as a Coleff-Herrera product.
We prove (a slightly stronger result than) Theorem \ref{epsilon-main} by induction over $p$.
Let $R^*$ denote the Coleff-Herrera product of some $R^j$:s with $j>p$ and let $U^*$ and $U^*_{\epsilon}$
denote the product of some $U^j$:s and $U^j_{\epsilon}$:s respectively, also with $j>p$ but only $j$:s not
occurring in $R^*$. We prove
\begin{equation*}
\big|R^1_{\epsilon}\wedge \cdots \wedge R^p_{\epsilon}\wedge R^*\wedge U^*_{\epsilon}-
R^1\wedge \cdots \wedge R^p\wedge R^*\wedge U^*\big|\lesssim \epsilon^{\omega},
\end{equation*} 
i.e., we prove Theorem \ref{epsilon-main} {\em on} the current $R^*$.
The induction start, $p=0$, follows immediately from \eqref{foljd}. If we add and subtract
$R^1_\epsilon\wedge \dots \wedge R^p_\epsilon \wedge R^*\wedge U^*$, the induction step follows easily
from \eqref{nablalikhet} (construed in setting of $\epsilon$-regularizations) and estimates like \eqref{foljd}.

\begin{proof}[Proof of Propositions \ref{epsilonpropp} and \ref{lambdapropp}]
We may assume that $\varphi$ has arbitrarily small support. Hence, we may assume that $Z$ is an analytic subset
of a domain $\Omega\subseteq \C^N$ and that all bundles are trivial,
and thus make the identification $f_j=(f_{j1},\ldots,f_{je_j})$, where $f_{ji}$ are holomorphic in $\Omega$. 
We choose a Hironaka resolution $\hat{Z}\rightarrow Z$ such that the pulled-back ideals $\langle\hat{f}_j\rangle$
are all principal, and moreover, so that in a fixed chart with coordinates $x$
on $\hat{Z}$ (and after a possible re-numbering), 
$\langle\hat{f}_j\rangle$ is generated by $\hat{f}_{j1}$ and $\hat{f}_{j1}=x^{\alpha_j}h_j$, where $h_j$ is holomorphic 
and non-zero. We then have
\begin{equation*}
|\hat{f}_j|^2=|\hat{f}_{j1}|^2\xi_j, \quad \hat{u}^j_{k_j}=v^j/\hat{f}^{k_j}_{j1},
\end{equation*}
where $\xi_j$ is smooth and positive and $v^j$ is a smooth (bundle valued) form. We thus get
\begin{equation*}
\debar \chi_j(|\hat{f}_j|^2/\epsilon_j)=
\tilde{\chi}_j(|\hat{f}_j|^2/\epsilon_j)\left(\frac{d\bar{\hat{f}}_{j1}}{\bar{\hat{f}}_{j1}}+
\frac{\debar \xi_j}{\xi_j}\right),
\end{equation*}
where $\tilde{\chi}_j(t)=t \chi_j'(t)$, and 
\begin{equation*}
\debar |\hat{f}_j|^{2\lambda_j}=\lambda_j |\hat{f}_j|^{2\lambda_j} \left(\frac{d\bar{\hat{f}}_{j1}}{\bar{\hat{f}}_{j1}}+
\frac{\debar \xi_j}{\xi_j}\right).
\end{equation*}
It follows that $\mathcal{I}(\epsilon)$ and $\Gamma(\lambda)$ are finite sums of integrals which
we without loss of generality can assume to be of the form
\begin{equation}\label{eq4}
\pm \int_{\C^n_x}\prod_1^p \tilde{\chi}_j^{\epsilon} \prod_{p+1}^q\chi_j^{\epsilon}
\bigwedge_1^m \frac{d\bar{\hat{f}}_{j1}}{\bar{\hat{f}}_{j1}}\wedge \bigwedge_{m+1}^p
\frac{\debar\xi_j}{\xi_j}\wedge \bigwedge_1^q\frac{v^j}{\hat{f}^{k_j}_{j1}}\wedge \varphi \rho,
\end{equation}
\begin{equation}\label{eq4'}
\pm \lambda_1\cdots \lambda_p \int_{\C^n_x}
\prod_1^q |\hat{f}_j|^{2\lambda_j}
\bigwedge_1^m \frac{d\bar{\hat{f}}_{j1}}{\bar{\hat{f}}_{j1}}\wedge \bigwedge_{m+1}^p
\frac{\debar\xi_j}{\xi_j}\wedge \bigwedge_1^q\frac{v^j}{\hat{f}^{k_j}_{j1}}\wedge \varphi \rho,
\end{equation}
where $\rho$ is a cutoff function.

\smallskip

Recall that $\hat{f}_{j1}=x^{\alpha_j}h_j$ and let $\mu$ be the number of vectors in a maximal 
linearly independent subset of $\{\alpha_1,\ldots,\alpha_m\}$; say that 
$\alpha_1,\ldots,\alpha_{\mu}$ are linearly independent. We then can define new holomorphic coordinates
(still denoted by $x$) so that $\hat{f}_{j1}=x^{\alpha_j}$, $j=1,\ldots,\mu$, see \cite[p.~46]{PCrelle} for details.
Then we get
\begin{eqnarray}\label{hack}
\bigwedge_1^m d\hat{f}_{j1} &=& \bigwedge_1^{\mu}dx^{\alpha_j} \wedge
\bigwedge_{\mu+1}^m(x^{\alpha_j}dh_j+h_jdx^{\alpha_j}) \\
&=& x^{\sum_{\mu+1}^m \alpha_j}\bigwedge_1^{\mu}dx^{\alpha_j}\wedge \bigwedge_{\mu+1}^m dh_j, \nonumber
\end{eqnarray}
where the last equality follows because $dx^{\alpha_1}\wedge \cdots \wedge dx^{\alpha_{\mu}}\wedge dx^{\alpha_j}=0$,
$\mu+1\leq j \leq m$, since $\alpha_1,\ldots,\alpha_{\mu},\alpha_j$ are linearly dependent.
From the beginning we could also have assumed that $\varphi=\varphi_1\wedge \varphi_2$, where
$\varphi_1$ is an anti-holomorphic $(n-\sum_1^q k_j +q-p)$-form and $\varphi_2$ is a (bundle valued)
$(n,0)$-test form on $Z$. We now define
\begin{equation*}
\Phi=\bigwedge_{\mu+1}^m \frac{d\bar{h}_j}{\bar{h}_j}\wedge \bigwedge_{m+1}^p\frac{\debar \xi_j}{\xi_j}\wedge
\bigwedge_1^q v^j \wedge \hat{\varphi}_1.
\end{equation*}
Using \eqref{hack} we can now write \eqref{eq4} and \eqref{eq4'} as
\begin{equation}\label{eq5}
\pm \int_{\C^n_x}\frac{\prod_1^p \tilde{\chi}_j^{\epsilon} \prod_{p+1}^q\chi_j^{\epsilon}}{\prod_1^q \hat{f}^{k_j}_{j1}}
\frac{d\bar{x}^{\alpha_1}}{\bar{x}^{\alpha_1}} \wedge \cdots \wedge \frac{d\bar{x}^{\alpha_{\mu}}}{\bar{x}^{\alpha_{\mu}}}
\wedge \Phi \wedge \hat{\varphi}_2\rho,
\end{equation}
\begin{equation}\label{eq5'}
\pm \lambda_1\cdots \lambda_p \int_{\C^n_x}
\frac{\prod_1^q |\hat{f}_j|^{2\lambda_j}}{\prod_1^q \hat{f}^{k_j}_{j1}}
\frac{d\bar{x}^{\alpha_1}}{\bar{x}^{\alpha_1}} \wedge \cdots \wedge \frac{d\bar{x}^{\alpha_{\mu}}}{\bar{x}^{\alpha_{\mu}}}
\wedge \Phi \wedge \hat{\varphi}_2\rho.
\end{equation}

\begin{lemma}\label{divlemma}
Let $\mathcal{K}=\{i;\, x_i \, \big| \, x^{\alpha_j}, \, \textrm{some} \,\, p+1\leq j \leq q\}$.
For any fixed $r\in \mathbb{N}$, one can replace $\Phi$ in \eqref{eq5} and \eqref{eq5'} by
\begin{equation*}
\Phi':=\Phi -
\sum_{J\subseteq \mathcal{K}}(-1)^{|J|}\sum_{k_1,\dots,k_{|J|} = 0}^{r+1}
\left.\frac{\partial^{|k|}\Phi}{\partial x_J^k}\right|_{x_J=0}
\frac{x_J^k}{k!}
\end{equation*}
without affecting the integrals. Moreover, for any $I\subseteq \mathcal{K}$, we have that 
$\Phi'\wedge \Lambda_{i\in I}(d\bar{x}_i/\bar{x}_i)$ is $C^r$-smooth.
\end{lemma}

We replace $\Phi$ by $\Phi'$ in \eqref{eq5} and \eqref{eq5'} and we 
write $d=d_{\mathcal{K}}+d_{\mathcal{K}^c}$, where $d_{\mathcal{K}}$ differentiates with respect to the 
variables $x_i$, $\bar{x}_i$ for $i\in \mathcal{K}$ and $d_{\mathcal{K}^c}$ differentiates with respect to the rest.
Then we can write 
$(d\bar{x}^{\alpha_1}/\bar{x}^{\alpha_1})\wedge \cdots \wedge (d\bar{x}^{\alpha_{\mu}}/\bar{x}^{\alpha_{\mu}})\wedge \Phi'$ 
as a sum of terms, which we without loss of generality can assume to be of the form 
\begin{equation*}
\frac{d_{\mathcal{K}^c}\bar{x}^{\alpha_1}}{\bar{x}^{\alpha_1}}\wedge \cdots \wedge 
\frac{d_{\mathcal{K}^c}\bar{x}^{\alpha_{\nu}}}{\bar{x}^{\alpha_{\nu}}}\wedge
\frac{d_{\mathcal{K}}\bar{x}^{\alpha_{\nu+1}}}{\bar{x}^{\alpha_{\nu+1}}}\wedge \cdots \wedge 
\frac{d_{\mathcal{K}}\bar{x}^{\alpha_{\mu}}}{\bar{x}^{\alpha_{\mu}}}\wedge \Phi'
\end{equation*}
\begin{equation*}
=\frac{d_{\mathcal{K}^c}\bar{x}^{\alpha_1}}{\bar{x}^{\alpha_1}}\wedge \cdots \wedge 
\frac{d_{\mathcal{K}^c}\bar{x}^{\alpha_{\nu}}}{\bar{x}^{\alpha_{\nu}}}\wedge
\Phi''\wedge d\bar{x}_{\mathcal{K}},
\end{equation*}
where $\Phi''$ is $C^r$-smooth and of bidegree $(0,n-\nu-|\mathcal{K}|)$ (possibly, $\Phi''=0$).
Thus, \eqref{eq5} and \eqref{eq5'} are finite sums of of integrals of the following type
\begin{equation}\label{eq6}
\int_{\C^n_x}\frac{\prod_1^p \tilde{\chi}_j^{\epsilon} \prod_{p+1}^q\chi_j^{\epsilon}}{\prod_1^q \hat{f}^{k_j}_{j1}}
\frac{d\bar{x}^{\alpha_1}}{\bar{x}^{\alpha_1}} \wedge \cdots \wedge \frac{d\bar{x}^{\alpha_{\nu}}}{\bar{x}^{\alpha_{\nu}}}
\wedge \psi \wedge d\bar{x}_{\mathcal{K}}\wedge dx,
\end{equation}
\begin{equation}\label{eq6'}
\lambda_1\cdots \lambda_p \int_{\C^n_x}
\frac{\prod_1^q |\hat{f}_j|^{2\lambda_j}}{\prod_1^q \hat{f}^{k_j}_{j1}}
\frac{d\bar{x}^{\alpha_1}}{\bar{x}^{\alpha_1}} \wedge \cdots \wedge \frac{d\bar{x}^{\alpha_{\nu}}}{\bar{x}^{\alpha_{\nu}}}
\wedge \psi \wedge d\bar{x}_{\mathcal{K}}\wedge dx,
\end{equation}
where $\psi$ is $C^r$-smooth and compactly supported.

\bigskip

We now first finish the proof of Proposition \ref{lambdapropp}. First of all, it is well known that
$\Gamma(\lambda)$ has a meromorphic continuation to $\C^q$. We have
\begin{equation*}
\frac{d\bar{x}^{\alpha_1}}{\bar{x}^{\alpha_1}} \wedge \cdots \wedge \frac{d\bar{x}^{\alpha_{\nu}}}{\bar{x}^{\alpha_{\nu}}}
\wedge d\bar{x}_{\mathcal{K}} =
\sum_{\stackrel{|I|=\nu}{I\subseteq \mathcal{K}^c}}C_I\frac{d\bar{x}_I}{\bar{x}_I}\wedge d\bar{x}_{\mathcal{K}}.
\end{equation*}
Let us assume that $I=\{1,\ldots,\nu\}\subseteq \mathcal{K}^c$ and consider the contribution to 
\eqref{eq6'} corresponding to this subset. This contribution equals
\begin{equation}\label{eq7'}
C_I\lambda_1\cdots \lambda_p \int_{\C^n_x}
\frac{|x^{\sum_1^q \lambda_j\alpha_j}|^2}{x^{\sum_1^q k_j\alpha_j}} \bigwedge_1^{\nu}\frac{d\bar{x}_j}{\bar{x}_j}\wedge
\Psi(\lambda,x)\wedge d\bar{x}_{\mathcal{K}}\wedge dx
\end{equation}
\begin{eqnarray*}
&=& \frac{C_I\prod_1^p\lambda_j}{\prod_{i=1}^{\nu}(\sum_1^q\lambda_j\alpha_{ji})}
\int_{\C^n_x} \frac{\bigwedge_{i=1}^{\nu}\debar |x_i|^{2\sum_1^q\lambda_j\alpha_{ji}}
\prod_{i=\nu+1}^n|x_i|^{2\sum_1^q\lambda_j\alpha_{ji}}}{x^{\sum_1^q k_j\alpha_j}}\wedge \\
& & \hspace{7cm}\wedge \Psi(\lambda,x)\wedge d\bar{x}_{\mathcal{K}}\wedge dx,
\end{eqnarray*}
where $\Psi(\lambda,x)=\psi(x)\prod_1^q(\xi_j^{\lambda_j}/h_j^{k_j})$.
It is well known (and not hard to prove, e.g., by integrations by parts as in \cite{MatsAB}, Lemma 2.1) that the 
{\em integral} on the right-hand side of \eqref{eq7'} has an analytic continuation in $\lambda$ to 
a neighborhood of $\cap_1^q\{\mathfrak{Re}\, \lambda_j \geq 0\}$.
(We thus choose $r$ in Lemma \ref{divlemma} large enough so that we can integrate by parts.)
If $p=0$, then the coefficient in front of 
the integral is to be interpreted as $1$ and Proposition \ref{lambdapropp} follows in this case. 
For $p>0$, we see that the poles of \eqref{eq7'}, and consequently
of $\Gamma(\lambda)$, in a neighborhood of $\cap_1^q\{\mathfrak{Re}\, \lambda_j \geq 0\}$ are along
hyperplanes
of the form $0=\sum_1^q\lambda_j\alpha_{ji}$, $1\leq i \leq \nu$. But if $j>p$ and $i\leq \nu$, then $\alpha_{ji}=0$
since $\{1,\ldots,\nu\}\subseteq \mathcal{K}^c=\{i;\, x_i \nmid x^{\alpha_j},\, \forall j=p+1,\ldots,q\}$.
Thus, the hyperplanes are of the form $0=\sum_1^p\lambda_j\alpha_{ji}$ and Proposition \ref{lambdapropp} is proved
except for the statement that at least for two $j$:s, the $\alpha_{ji}$ are non-zero. However, we see from 
\eqref{eq7'} that if for some $i$ we have $\alpha_{ji}=0$ for all $j$ but one, then the appearing $\lambda_j$ in
the denominator will be canceled by the numerator. Moreover, we may assume that 
the constant $C_I=\det (\alpha_{ji})_{1\leq i,j\leq \nu}$ is non-zero which implies that we cannot have any
$\lambda_j^2$ in the denominator. 

\bigskip

We now prove Proposition \ref{epsilonpropp}. Consider \eqref{eq6}. We have that $\alpha_1,\ldots,\alpha_{\nu}$
are linearly independent so we may assume that $A=(\alpha_{ij})_{1\leq i,j\leq \nu}$ is invertible
with inverse $B=(b_{ij})$. We make the non-holomorphic change of variables \eqref{varbyte}, where the ``$q$'' of 
\eqref{varbyte} now should be understood as $\nu$. Then we get $x^{\alpha_j}=y^{\alpha_j}\eta_j$, where 
$\eta_j>0$ and smooth and $\eta_j^2=1/\xi_j$, $j=1,\ldots,\nu$. 
Hence, $|\hat{f}_j|^2=|y^{\alpha_j}|^2$, $j=1,\ldots,\nu$.
Expressed in the $y$-coordinates we get that 
$\Lambda_1^{\nu}(d\bar{x}^{\alpha_j}/\bar{x}^{\alpha_j})\wedge \psi \wedge d\bar{x}_{\mathcal{K}}\wedge dx$ 
is a finite sum of terms of the form 
\begin{equation}\label{hack2}
\frac{d\bar{y}^{\alpha_1}}{\bar{y}^{\alpha_1}} \wedge \cdots \wedge \frac{d\bar{y}^{\alpha_{\nu'}}}{\bar{y}^{\alpha_{\nu'}}}
\wedge \bar{y}_{\mathcal{K}'}\, d\bar{y}_{\mathcal{K}''} \wedge \psi_1,
\end{equation}
where $\nu'\leq \nu$, $\psi_1$ is a $C^r$-smooth compactly supported form, and
$\mathcal{K}'$ and $\mathcal{K}''$ are disjoint sets such that $\mathcal{K}'\cup \mathcal{K}''=\mathcal{K}$.
In order to give a contribution to \eqref{eq6} we see that $\psi_1$ must contain $dy$. 
In \eqref{hack2} we write $d=d_{\mathcal{K}}+d_{\mathcal{K}^c}$, and arguing as we did
immediately after Lemma \ref{divlemma}, \eqref{hack2} is a finite sum of terms of the form
\begin{equation*}
\frac{d\bar{y}^{\alpha_1}}{\bar{y}^{\alpha_1}} \wedge \cdots \wedge \frac{d\bar{y}^{\alpha_{\nu''}}}{\bar{y}^{\alpha_{\nu''}}}
\wedge \psi_2 \wedge d\bar{y}_{\mathcal{K}}\wedge dy,
\end{equation*}
where $\nu''\leq \nu$ and $\psi_2$ is $C^r$-smooth and compactly supported.
With abuse of notation we thus have that \eqref{eq6} is a finite sum of integrals of the form
\begin{equation}\label{eq7}
\int_{\C^n_x}\frac{\prod_1^p \tilde{\chi}_j^{\epsilon} \prod_{p+1}^q\chi_j^{\epsilon}}{\prod_1^q \hat{f}^{k_j}_{j1}}
\frac{d\bar{y}^{\alpha_1}}{\bar{y}^{\alpha_1}} \wedge \cdots \wedge \frac{d\bar{y}^{\alpha_{\nu}}}{\bar{y}^{\alpha_{\nu}}}
\wedge \psi \wedge d\bar{y}_{\mathcal{K}}\wedge dy
\end{equation}
\begin{equation*}
=\int_{\C^n_x}\frac{\bigwedge_1^{\nu}d\chi_j^{\epsilon}
\prod_{\nu+1}^p \tilde{\chi}_j^{\epsilon} \prod_{p+1}^q\chi_j^{\epsilon}}{y^{\sum_1^q k_j\alpha_j}}
\wedge \Psi \wedge d\bar{y}_{\mathcal{K}}\wedge dy,
\end{equation*}
where $\Psi$ is a $C^r$-smooth compactly supported $(n-|\mathcal{K}|-\nu)$-form; the equality follows since
$\chi_j^{\epsilon}=\chi_j(|y^{\alpha_j}|^2/\epsilon_j)$, $j=1,\ldots,\nu$. Now, \eqref{eq7} is essentially equal
to equation (24) of \cite{JebHs} and the proof of Proposition \ref{epsilonpropp} is concluded as in the 
proof of Proposition 8 in \cite{JebHs}.
\end{proof}

\bigskip

\begin{proof}[Proof of Lemma \ref{divlemma}]
The proof is similar to the proof of Lemma 9 in \cite{JebHs} but some modifications have to be done.
First, it is easy to check by induction over $|\mathcal{K}|$ that 
$\Phi'\wedge \Lambda_{i\in I}(d\bar{x}_i/\bar{x}_i)$ is $C^r$-smooth for any $I\subseteq \mathcal{K}$; for
$|\mathcal{K}|=1$ this is just Taylor's formula for forms. It thus suffices to show that
\begin{equation*}
d\bar{x}^{\alpha_1}\wedge \cdots \wedge d\bar{x}^{\alpha_{\mu}}\wedge \left.\frac{\partial^{|k|} \Phi}{\partial x_I^k}\right|_{x_I=0}=0,\quad
\forall I\subseteq \mathcal{K}, \, k=(k_{i_1},\ldots,k_{i_{|I|}}).
\end{equation*}
To show this, fix an $I\subseteq \mathcal{K}$ and let $L=\{j; \, x_i \nmid x^{\alpha_j}\,\, \forall i\in I\}$. 
Say for simplicity that 
\begin{equation*}
L=\{1,\ldots,\mu',\mu+1,\ldots,m',m+1,\ldots,p',p+1,\ldots,q'\},
\end{equation*}
where $\mu'\leq \mu$, $m'\leq m$, $p'\leq p$, and $q'<q$. The fact that $q'<q$ follows from the definitions
of $\mathcal{K}$, $I$, and $L$.

Consider, on the base variety $Z$, the smooth form
\begin{equation*}
F=\bigwedge_1^{\mu'}d\bar{f}_{j1} \bigwedge_{\mu+1}^{m'}d\bar{f}_{j1}
\bigwedge_{m+1}^{p'}(|f_{j1}|^2\debar |f_j|^2-\debar |f_{j1}|^2 |f_j|^2)
\bigwedge_{j\in L}|f_j|^{2k_j}u^j_{k_j} \wedge \varphi_1.
\end{equation*}
It has bidegree $(0,n-\sum_{j\in L^c}k_j +q-q')$ so $F$ 
has a vanishing pullback to $\cap_{j\in L^c}\{f_j=0\}$ since this set has dimension 
$n-\sum_{j\in L^c}e_j < n-\sum_{j\in L^c}k_j +q-q'$ by our assumption about complete intersection.
Thus, $\hat{F}$ has a vanishing pullback to $\{x_I=0\}\subseteq \cap_{j\in L^c}\{\hat{f}_j=0\}$.
In fact, this argument shows that 
\begin{equation}\label{Feq}
\hat{F}=\sum\phi_j,
\end{equation}
where the $\phi_j$ are smooth linearly independent forms such that each $\phi_j$ is divisible by 
$\bar{x}_i$ or $d\bar{x}_i$ for some $i\in I$. (It is the pull-back to $\{x_I=0\}$ of the anti-holomorphic
differentials of $\hat{F}$ that vanishes.)
For the rest of the proof we let $\sum \phi_j$ 
denote such expressions and we note that they are invariant under holomorphic differential operators.
Computing $\hat{F}$ we get
\begin{equation*}
\hat{F}=\prod_{m+1}^{p'} |\hat{f}_{j1}|^4\prod_{j\in L}\frac{|\hat{f}_j|^{2k_j}}{\hat{f}_{j1}^{k_j}}
\bigwedge_1^{\mu'}d\bar{x}^{\alpha_j} \bigwedge_{\mu+1}^{m'}d(\bar{x}^{\alpha_j}\bar{h}_j)
\bigwedge_{m+1}^{p'}\debar \xi_j
\bigwedge_{j\in L}v^j \wedge \hat{\varphi}_1.
\end{equation*}
The ``coefficient'' $\prod_{m+1}^{p'} |\hat{f}_{j1}|^4\prod_{j\in L}(|\hat{f}_j|^{2k_j}/\hat{f}_{j1}^{k_j})$
does not contain any $\bar{x}_i$ with $i\in I$ so we may divide \eqref{Feq} by it 
(recall that the $\phi_j$ are linearly independent) and we obtain
\begin{eqnarray*}
\sum\phi_j &=&
\bigwedge_1^{\mu'}d\bar{x}^{\alpha_j} \bigwedge_{\mu+1}^{m'}d(\bar{x}^{\alpha_j}\bar{h}_j)
\bigwedge_{m+1}^{p'}\debar \xi_j
\bigwedge_{j\in L}v^j \wedge \hat{\varphi}_1 \\
&=& \prod_{\mu+1}^{m'}\bar{x}^{\alpha_j} \bigwedge_1^{\mu'}d\bar{x}^{\alpha_j} \bigwedge_{\mu+1}^{m'}d\bar{h}_j
\bigwedge_{m+1}^{p'}\debar \xi_j
\bigwedge_{j\in L}v^j \wedge \hat{\varphi}_1 \\
& & + \bigwedge_1^{\mu'}d\bar{x}^{\alpha_j} \wedge \sum_{\mu+1}^{m'}d\bar{x}^{\alpha_j}\wedge \tau_j
\end{eqnarray*}
for some $\tau_j$. We multiply this equality with 
\begin{equation*}
\bigwedge_{m'+1}^md\bar{h}_j \bigwedge_{p'+1}^p \debar \xi_j \bigwedge_{j\in L^c}v^j/\left(\prod_{\mu+1}^m\bar{h}_j
\prod_{m+1}^p\xi_j\right)
\end{equation*}
and get
\begin{equation*}
\prod_{\mu+1}^{m'}\bar{x}^{\alpha_j} \bigwedge_1^{\mu'}d\bar{x}^{\alpha_j}\wedge \Phi+
\bigwedge_1^{\mu'}d\bar{x}^{\alpha_j} \wedge \sum_{\mu+1}^{m'}d\bar{x}^{\alpha_j}\wedge \tau_j
=\sum\phi_j
\end{equation*}
for some new $\tau_j$. We apply the operator $\partial^{|k|}/\partial x_I^k$ to this equality and then we
pull back to $\{x_I=0\}$, which makes the right-hand side vanish;
(we construe however the result in $\C^n_x$). Finally, taking the exterior product with
$\Lambda_{\mu'+1}^{\mu}d\bar{x}^{\alpha_j}$, which will make each term in under the summation sign on the left-hand
side vanish, we arrive at
\begin{equation*}
\prod_{\mu+1}^{m'}\bar{x}^{\alpha_j} \bigwedge_1^{\mu}d\bar{x}^{\alpha_j}\wedge 
\frac{\partial^{|k|}\Phi}{\partial x_I^k}\big|_{x_I=0}=0
\end{equation*} 
and we are done.
\end{proof}

\section*{Acknowledgments}

We would like to thank the anonymous referee for valuable comments regarding the presentation of the article.


\begin{thebibliography}{9999}

\def\listing#1#2#3{{\sc #1}:\ {\it #2},\ #3.}


\bibitem{MatsAB}\listing{M. Andersson}
    {Residue currents and ideals of holomorphic functions}
    {Bull. Sci. Math. {\bf 128} (2004), 481--512}

\bibitem{MatsAM}\listing{M. Andersson}
    {Explicit versions of the Brian\c con-Skoda theorem with variations}
    {Michigan Math. J. {\bf 54} (2006),  no. 2, 361--373}

\bibitem{MatsAArk}\listing{M. Andersson}
    {A residue criterion for strong holomorphicity}
    {Ark. Mat. {\bf 48} (2010), 1--15}

\bibitem{AS}\listing{M. Andersson, H. Samuelsson}
    {Weighted Koppelman formulas and the $\debar$-equation on an analytic space}
    {J. Funct. Anal. {\bf 261} (2011), 777--802}
    
\bibitem{AS2}\listing{M. Andersson, H. Samuelsson}
    {A Dolbeault-Grothendieck lemma on complex spaces via Koppelman formulas}
    {Invent. Math. {\bf 190} (2012), no. 2, 261--297}

\bibitem{ASWY}\listing{M. Andersson, H. Samuelsson, E. Wulcan, A. Yger}
    {Nonproper intersection theory and positive currents I, local aspects}
    {Preprint, G{\"o}teborg, available at arXiv:1009.2458 [math.CV, math.AG]}

\bibitem{AWCrelle}\listing{M. Andersson, E. Wulcan}
    {Decomposition of residue currents}
    {J. Reine Angew. Math. {\bf 638} (2010), 103--118}

\bibitem{At}\listing{M. F. Atiyah}
    {Resolution of singularities and division of distributions}
    {Comm. Pure Appl. Math. {\bf 23} (1970) 145--150}

\bibitem{BaMa}\listing{D. Barlet, H-M. Maire}
    {Asymptotic expansion of complex integrals via Mellin transform}
    {J. Funct. Anal. {\bf 83} (1989),  no. 2, 233--257}

\bibitem{BGVY}\listing{C. A. Berenstein, R. Gay, A. Vidras, A. Yger}
    {Residue currents and Bezout identities}
    {Progr. Math. vol. 114, Birkh\"auser, Basel, 1993}

\bibitem{BGY}\listing{C. A. Berenstein, R. Gay, A. Yger}
    {Analytic continuation of currents and division problems}
    {Forum Math.  1  (1989),  no. 1, 15--51}

\bibitem{BYJAM}\listing{C. A. Berenstein, A. Yger}
    {Green currents and analytic continuation}
    {J. Anal. Math. {\bf 75} (1998), 1--50}

\bibitem{BePa}\listing{B. Berndtsson, M. Passare}
    {Integral formulas and an explicit version of the fundamental principle}
    {J. Funct. Anal.  84  (1989),  no. 2, 358--372}

\bibitem{BG}\listing{I. N. Bernstein, S. I. Gel'fand}
    {Meromorphic property of the functions $P^{\lambda}$}
    {Functional Anal. Appl. {\bf 3} (1969) 68--69}

\bibitem{JebAbel}\listing{J.-E. Bj\"ork}
    {Residues and $\mathscr{D}$-modules}
    {The legacy of Niels Henrik Abel, 605--651, Springer, Berlin, 2004}

\bibitem{JebHs}\listing{J.-E. Bj\"ork, H. Samuelsson}
    {Regularizations of residue currents}
    {J. Reine Angew. Math. {\bf 649} (2010), 33--54}

\bibitem{CH}\listing{N.R. Coleff, M. E. Herrera}
    {Les courants r\`esiduels associ\'es \`a une forme meromorphe}
    {Lecture Notes in Mathematics, 633, Springer, Berlin, 1978}

\bibitem{DS}\listing{A. Dickenstein, C. Sessa}
    {Canonical representatives in moderate cohomology}
    {Inventiones Math. {\bf 80} (1985) 417-434}

\bibitem{HePo}\listing{G. Henkin, P. Polyakov}
    {The Grothendieck-Dolbeault lemma for complete intersections}
    {C. R. Acad. Sci. Paris S\'er. I Math. {\bf 308} (1989),  no. 13, 405--409}

\bibitem{HL}\listing{M. Herrera, D. Lieberman}
    {Residues and Principal Values on Complex Spaces}
    {Math. Ann. {\bf 194} (1971) 259-294}

\bibitem{Hiro}\listing{H. Hironaka}
    {Desingularization of complex-analytic varieties}
    {Actes, Congr\`{e}s Intern.\ Math., 1970. Tome 2, 627--631}

\bibitem{larkang}\listing{R.\ L\"{a}rk\"{a}ng}
    {Residue currents associated with weakly holomorphic functions}
    {Ark. Mat. {\bf 50} (2012), no. 1, 135--164}

\bibitem{PCrelle}\listing{M. Passare}
    {A calculus for meromorphic currents}
    {J. Reine Angew. Math. {\bf 392} (1988) 37-56}

\bibitem{Pdr}\listing{M. Passare}
    {Residues, currents, and their relation to ideals of holomorphic functions}
    {Math. Scand. {\bf 62} (1988) 75--152}
    
\bibitem{Plambda}\listing{M. Passare}
    {Courants m\'{e}romorphes et \'{e}galit\'{e} de la valeur principale et de la partie finie. (French) }
    {Lecture Notes in Math., {\bf 1295}, Springer, Berlin, 1987, pp. 157--166}

\bibitem{PTex}\listing{M. Passare, A. Tsikh}
    {Defining the residue of a complete intersection}
    {Complex analysis, harmonic analysis and applications (Bordeaux, 1995), 250--267, Pitman Res. Notes Math. Ser., 347, Longman, Harlow, 1996}

\bibitem{PTCanad}\listing{M. Passare, A. Tsikh}
    {Residue integrals and their Mellin transforms}
    {Canad. J. Math.  47  (1995),  no. 5, 1037--1050}

\bibitem{PTY}\listing{M. Passare, A. Tsikh, A. Yger}
    {Residue currents of the Bochner-Martinelli type}
    {Publ. Mat. {\bf 44} (2000) 85-117}

\bibitem{hasamJFA}\listing{H. Samuelsson}
    {Regularizations of products of residue and principal value currents}
    {J. Funct. Anal, {\bf 239(2)} (2006), 566--593}

\bibitem{HasamArkiv}\listing{H. Samuelsson}
    {Analytic continuation of residue currents}
    {Ark. Mat. {\bf 47} (2009) 127--141}

\bibitem{TsikhBook}\listing{A. Tsikh}
    {Multidimensional residues and their applications}
    {Translations of Mathematical Monographs, 103. American Mathematical Society, Providence, RI, 1992}

\bibitem{WArkiv}\listing{E. Wulcan}
    {Products of residue currents of Cauchy-Fantappi\`e-Leray type}
    {Ark. Mat. {\bf 45} (2007) 157--178}

\bibitem{Y}\listing{A. Yger}
    {Formules de division et prolongement m\'eromorphe}
    {S\'eminaire d'Analyse P. Lelong -- P. Dolbeault -- H. Skoda, Ann\'ees 1985/1986, Lecture Notes in Math. vol. 1295, Springer, Berlin,
    1987, 226--283}

\end{thebibliography}
\end{document}